\documentclass[11pt]{amsart}
\usepackage{amssymb, amsmath, amsthm}
\usepackage{amsrefs}
\usepackage{graphicx}
\usepackage[final]{hyperref}
\usepackage[a4paper, centering]{geometry}
\usepackage{color}
\usepackage{graphicx}

\newtheorem{theorem}{Theorem}
\newtheorem{proposition}[theorem]{Proposition}
\newtheorem{lemma}[theorem]{Lemma}

\theoremstyle{definition}
\newtheorem{definition}[theorem]{Definition}
\theoremstyle{remark}
\newtheorem{remark}[theorem]{Remark}
\newtheorem{example}[theorem]{Example}
\def\R{\mathbb{R}}

\def\N{\mathbb{N}}
\def \e{\varepsilon}

\def\pscal#1#2{\left\langle#1,\,#2\right\rangle}
\def\blambda{\boldsymbol{t}}
\def\eigen{\overline{\lambda}}
\def\ubar{\overline{u}}
\def\osubjet{J^{2, -}_{\Omega}}
\def\osuperjet{J^{2, +}_{\Omega}}

\newcommand*{\infconv}[2][k]{{(#2_0\sharp\cdots\sharp #2_{#1})}_{\blambda}}

\newcommand*{\clset}{\mathcal{A}^n}

\DeclareMathOperator{\dist}{dist}

\begin{document}

 %amsart format
\title[The Brunn--Minkowski inequality]%
{The Brunn--Minkowski inequality\\
for the principal eigenvalue of\\
fully nonlinear homogeneous elliptic operators}%
\author[G.~Crasta, I.~Fragal\`a]{Graziano Crasta,  Ilaria Fragal\`a}
\address[Graziano Crasta]{Dipartimento di Matematica ``G.\ Castelnuovo'', Univ.\ di Roma I\\
P.le A.\ Moro 5 -- 00185 Roma (Italy)}
\email{crasta@mat.uniroma1.it}

\address[Ilaria Fragal\`a]{
Dipartimento di Matematica, Politecnico\\
Piazza Leonardo da Vinci, 32 --20133 Milano (Italy)
}
\email{ilaria.fragala@polimi.it}



\date{July 28, 2019}

\begin{abstract} 
We prove that the principal eigenvalue of any fully nonlinear homogeneous elliptic operator which fulfills a very simple convexity assumption  satisfies a Brunn-Minkowski type inequality on the class of open bounded sets in $\mathbb{R}^n$ satisfying a uniform exterior sphere condition. 
In particular the result applies to the (possibly normalized) $p$-Laplacian, and to the minimal Pucci operator. The  proof is inspired by the approach introduced by Colesanti for the principal frequency of the Laplacian within the class of convex domains, and relies on a generalization of the convex envelope method by Alvarez-Lasry-Lions.
We also deal with the existence and log-concavity of positive viscosity eigenfunctions. 
\end{abstract}

\maketitle

\section{Introduction}
In its classical formulation, the Brunn-Minkowski inequality
states that the volume functional, raised to the power $1/n$, is concave on the class $\mathcal K ^n$ of convex bodies in the $n$-dimensional Euclidean space. 
Specifically, for every pair
$K_0, K _1$  of nonempty convex compact subsets of $\R^n$ and every $t \in [0,1]$,
denoting by $( 1-t) K _0 + t K _1$  the set of points of the form $(1-t)x+ t y$ for $x \in K _0$ and $y \in K _1$, and by $V(\cdot)$ the $n$-dimensional Lebesgue measure,
it holds
\begin{equation}
\label{f:classic}
V ^ {1/n} \big ( ( 1-t) K _0 + t K _1\big ) \geq (1-t) V ^ {1/n}(K_0) + t V ^ {1/n} (K_1)\, ,
\end{equation}
with equality sign if and only if $K _0$ and $K _1 $ are homothetic.

Named after Brunn, who firstly proved it in dimension $2$ and $3$ \cite{brunn1, brunn2}, and Minkowski, who  shortly afterwards gave a full analytic proof in $n$-dimensions and characterized the equality case \cite{mink}, in the last century this fundamental inequality has been proved and generalized in many different ways by an impressive list of mathematicians, including Hilbert \cite{hilbert}, Bonnesen \cite{bonn}, Kneser-Suss \cite{knesuss}, Blaschke \cite{blaschke},  Hadwiger \cite{had}, Knothe \cite{knothe}, Dinghas \cite{ding}, MacCann \cite{mccann}, McMullen \cite{mullen}, Ball \cite{ball},  Klain \cite{klain}.

It is not conceivable to give here an idea about the impact of  Brunn-Minkowski inequality in both Analysis and Geometry, and in their interplay. We limit ourselves to refer to Chapter 7  in the treatise \cite{Sch} by Schneider, which includes a lot of historical and bibliographical notes, and to the excellent survey paper \cite{gard} by Gardner, from which we quote:
``In a sea of mathematics, the Brunn-Minkowski inequality appears like an octopus, tentacles reaching far and wide, its shape and color changing as it roams from one area to the next.''

Aim of this paper is to reveal a new tentacle of this fascinating creature, which gets as far as the viscosity theory of  nonlinear PDEs, by proving the validity of a Brunn-Minkowski  type inequality for the principal frequency of fully non-linear homogeneous elliptic operators. 

As a starting point to introduce our results, we recall that Brunn-Minkowski inequality has been generalized, in a suitable form, to several functionals other than volume. They include not only geometric quantities (such as quermassintegrals \cite[Section 7.4]{Sch}), but also some energies from physics and calculus of variations. 
To be more precise, following \cite{CoCuSa}, we say that a functional $\Phi$ which is invariant under rigid motions and homogeneous of degree $\gamma \neq 0$  on $\mathcal K ^n$ satisfies a Brunn-Minkowski type inequality if,  by analogy to \eqref{f:classic}, it holds
\begin{equation}\label{f:variational}
\Phi ^ {1/\gamma } \big ( ( 1-t) K _0 + t K _1\big ) \geq (1-t) \Phi  ^ {1/\gamma }(K_0) + t \Phi  ^ {1/\gamma } (K_1)\, .
\end{equation}
The most significant choices of functionals $\Phi$ for which the above inequality has been proved are:
the principal frequency of the Laplacian (see Brascamp-Lieb \cite{BrLi}),
the torsional rigidity (see Borell \cite{bor1}),
the Newtonian capacity (see Borell \cite{bor2} and Caffarelli-Jerison-Lieb \cite{CJL}),
the logarithmic capacity  and a $n$-dimensional version of it (see Borell \cite{bor3} and Colesanti-Cuoghi \cite{CoCu}),
the $p$-capacity (see Colesanti-Salani \cite{CoSa}), the first eigenvalue of the $p$-Laplacian and the $p$-torsional rigidity  (see Colesanti-Cuoghi-Salani \cite{CoCuSa}),  the first eigenvalue of the Monge-Amp\`ere operator (see Salani \cite{Sal1}), 
the Bernoulli constant (see Bianchini-Salani \cite{BS09}),
the Hessian eigenvalue in three-dimensional convex domains (see Liu-Ma-Xu \cite{LiMaXu}),  functionals related to Hessian equations (see Salani \cite{Sal2}).
For large part of these results, a nice account can be found in the paper \cite{Col2005} by Colesanti.

Regarding this spectrum of extensions of the Brunn-Minkowski inequality, we wish to draw attention on  the class of domains where the inequality is known to work. Actually, the validity of inequality \eqref{f:classic}  for the volume functional goes far beyond the class of convex bodies: it has been extended to all measurable sets; a short and elegant proof due to Hadwiger-Ohmann \cite{hadoh} can be found in the above mentioned survey paper by Gardner.  
In spite, to our knowledge,
for all the functionals $\Phi$ mentioned above
the validity of inequality \eqref{f:variational} has been established only within convex bodies, exception made for the first eigenvalue of the Laplacian
and the torsional rigidity, for which the inequality is known to hold for all open bounded domains with sufficiently regular boundary.

It is now time to present the new family of Brunn-Minkowski type inequalities we obtain in this paper.
Given an open bounded domain $\Omega$ in $\R ^n$, 
we consider the following  eigenvalue problem for a fully nonlinear, degenerate elliptic,  homogeneous operator:
\begin{equation}\label{f:pb} 
\begin{cases}
 F(\nabla  u, D ^ 2 u ) = \lambda  |u|^{\alpha} u & \text{ in } \Omega
 \\  \noalign{\medskip}
 u = 0 & \text{ on } \partial \Omega\,.
 \end{cases}
 \end{equation}
Here $F\colon(\R^n\setminus\{0\})\times S^n\to\R$ 
is a continuous function satisfying,  for every  $\xi \in \R^n \setminus \{ 0 \}$ and $X$ in the space $S^n$  of symmetric real matrices,  the following conditions:
  \begin{itemize}
 \item[(H1)] \textit{Homogeneity}: for some $\alpha > -1$ and every  $(t, \mu) \in (\R \setminus \{ 0 \} )\times \R$,
  \[
F (t\xi, \mu X) = |t| ^ \alpha \mu F ( \xi, X);
\]
  \item[(H2)] \textit{Uniform ellipticity}: for some $C\geq c>0$ and every  $Y$ in the space $S ^n _+$ of positive semidefinite  symmetric matrices,
 $$c |\xi| ^ \alpha {\rm tr} (Y) \leq F (\xi, X) - F (\xi, X+Y) \leq C |\xi|^ \alpha {\rm tr} (Y).$$
\end{itemize}

For any operator satisfying (H1)-(H2), inspired by the celebrated work by Berestycki, Nirenberg and Varadhan \cite{BNV},  Birindelli and Demengel introduced in \cite{BiDe2006} the  principal eigenvalue $\overline \lambda(\Omega)$ as
$$\overline \lambda (\Omega):= \sup \Big \{ \lambda \in \R \ :\ \exists u >0 \text{ in } \Omega \text{ viscosity super-solution to the pde in } \eqref{f:pb}\Big \} \,; $$
here the notion of viscosity super-solution has to be meant as specified in Section~\ref{sec:lambda}  below.

The bibliography related to the eigenvalue problem for fully nonlinear second order operators is very wide. With no attempt of completeness, we limit ourselves to quote
Birindelli-Demengel for many related works including \cite{BiDe2004,BiDe2007,BiDe2007b,BirDem2010}, Ikoma-Ishii for the computation of eigenvalues on balls \cite{II12, II15}, Armstrong \cite{arm}, 
Berestycki-Capuzzo Dolcetta-Porretta-Rossi \cite{BCPR}, and Quaas-Sirakov \cite{QS} for related maximum principles, Berestycki-Rossi for the case of unbounded domains \cite{BeRo}, Kawohl and different coauthors for the case of the game theoretic $p$-Laplacian \cite{K0,KH,KKK,K11,BK18}  (see also our recent joint work \cite{CFK}),   Juutinen for the case of the normalized infinity Laplacian \cite{Juut07}, Busca-Esteban-Quaas for the case of Pucci operators \cite{BEQ}.

As far as we know, there is no previous attempt to prove that 
the Brunn-Minkowski inequality holds true for the principal eigenvalue of a fully nonlinear operator. Our main result states that this is indeed the case as soon as the operator  enjoys, besides (H1)-(H2), the following condition

  \begin{itemize}
   \item[(H3)] {\it Convexity:} for every $\xi \in \R ^n \setminus \{ 0 \}$,
\[
X \mapsto F(\xi, X) \text{ is convex on } S ^n \,,
\]
 \end{itemize}
and the involved domains belong to the class
\begin{equation}
\label{f:sets}
\clset := \left\{ 
\begin{matrix}
\text{open bounded connected Lipschitz domains of $\R^n$} 
\\ 
\text{satisfying a uniform exterior sphere condition}
\end{matrix}
\right\}\,.
\end{equation}
We remark that this class is closed with respect to
the Minkowski addition of sets.

\begin{theorem}[Brunn-Minkowski inequality]\label{t:BM}
If $F$ satisfies conditions (H1)-(H2)-(H3), 
for every pair of domains $\Omega_0, \Omega_1\in\clset$,
and every $t\in [0,1]$, it holds
\begin{equation}\label{f:BMo}
\eigen \big ((1-t) \Omega _0 + t \Omega _1 \big )^{-1/(\alpha+2)}
\geq
(1-t)\eigen(\Omega_0)^{-1/(\alpha+2)} + t\, \eigen(\Omega_1)^{-1/(\alpha+2)}\, .
\end{equation}
\end{theorem}

\bigskip
We emphasize that 
the class $\clset$ contains all bounded open sets 
which are convex or of class $C^2$, but
domains in $\clset$ do not need to be convex, nor of class $C ^ 2$. 
In particular, for the first eigenvalue of the $p$-Laplacian, Theorem \ref{t:BM} extends  to 
domains in $\clset$ the Brunn-Minkowski inequality proved for $C^2$ convex bodies by Colesanti-Cuoghi-Salani \cite{CoCuSa}. 
Besides the $p$-Laplacian, a list of further relevant operators fitting the assumptions of Theorem \ref{t:BM} is postponed at the end of this section. 

\smallskip 
The reason why we work on  the class $\clset$ is that, for such domains, 
we are able to prove the existence of positive viscosity eigenfunctions, until now known only  for $C ^ 2$ domains
(see \cite{BiDe2006,BiDe2007b}). 
 This  side result, which may have its own interest, is given in 
Section~\ref{ss:existence} (see Theorem \ref{t:approx}): it 
is derived as a by-product of a global H\"older estimate (see Proposition \ref{p:holder}), which in turn is obtained
via a barrier argument, adapted from Birindelli-Demengel, involving the distance from the boundary.

\smallskip

Our approach to obtain Theorem~\ref{t:BM} can be synthetically  defined as a synergy between the method  introduced by Colesanti in \cite{Col2005} to obtain the Brunn-Minkowski inequality for the first eigenvalue of the Laplacian for convex domains, and the method introduced by Alvarez, Lasry and Lions in \cite{ALL} to obtain the convexity of viscosity solutions to  second order fully nonlinear elliptic equations with state constraint boundary conditions.
We also point out that in the paper \cite{IS16} parabolic problems are considered under a close perspective, working on possibly non-convex domains, yet still with classical solutions; more specifically, using Lemma 3.1 in \cite{IS16} it can be realized that the general theory previously developed by Salani  in \cite{SalAIHP}  (covering for instance the case of the Pucci operator) can be extended to non-convex domains.

Roughly speaking, the proof of Theorem \ref{t:BM}  goes as follows.  
The key point is that, in order to
prove  the inequality \eqref{f:variational} for $\Phi (\cdot) = \overline \lambda (\cdot)$, it is enough to construct a sub-solution to the corresponding eigenvalue problem on the domain $(1-t) \Omega _0 + t \Omega _1$.
In case of the Laplacian, this assertion relies on the variational characterization of the eigenvalue as minimum of the
Rayleigh quotient. In our fully nonlinear setting, though there is no variational interpretation of the eigenvalue,
the same principle remains true thanks to a maximum principle proved by Birindelli-Demengel (see Theorem \ref{t:birdem} below).
 Then the next step is how to construct a sub-solution.  
 To that aim the idea  is to look at the transformed equation satisfied by (minus) the logarithms of the eigenfunctions (which on convex domains are known to be convex functions \cite{BrLi, CaSp}), consider (minus) the infimal convolution between these logarithms, and take its exponential.
 In case of the Laplacian, the function thus obtained  turns out to be a sub-solution essentially  because the infimal convolution linearizes the Fenchel transform, and the map $M \mapsto {\rm tr} ( M ^ {-1})$ is convex on the family of positive definite matrices.
In our fully non-linear setting, we still consider the function constructed in the same way, but in order to show that it is a sub-solution  
we  have to adopt a different procedure.  Indeed, since we do not have enough regularity information on the eigenfunctions,
we cannot write pointwise Hessians; moreover, since
we want to get rid of the convexity assumptions on the domains, we cannot
exploit the log-concavity of eigenfunctions.  To overcome these difficulties, we set up a generalization of
the method introduced by Alvarez-Lasry-Lions  in order to show that the convex envelope is a sub-solution, the difference being that we work with a family of {\it distinct} functions on {\it distinct, possibly non-convex}, domains (compare Propositions~\ref{p:ALL0} and~\ref{p:concave} below  respectively with  Propositions~1 and~3 in \cite{ALL}).
We remark that similar techniques have been used in the above mentioned
paper \cite{SalAIHP} by Salani, where the author has introduced
a very general theory for Brunn-Minkowski
inequalities for functionals related to elliptic PDEs,
for a very general class of nonlinear operators.
Yet, the effective applicability of the results in \cite{SalAIHP}
is limited by the fact that only classical solutions are considered.

\smallskip
Let us point out that at present we are not able to push over our viscosity approach in order to deal with the equality case in Theorem~\ref{t:BM}. We address such characterization as an interesting open problem, which seems to be quite delicate. Actually, for a lot of Brunn-Minkowski type inequalities,  the characterization of the equality case is 
still open, especially when dealing with non-convex domains.
The case of the first eigenvalue of the Laplacian is emblematic in this respect:
since Brascamp-Lieb \cite{BrLi}, the inequality \eqref{f:variational} is known to hold for all compact, connected domains having sufficiently regular boundary, but the equality case has been settled only
forty years later by Colesanti \cite{Col2005}, and his approach works  just for convex domains.

\smallskip
On the other hand, as a companion result to Theorem \ref{t:BM}, we are able to establish the 
log-concavity of positive viscosity eigenfunctions. 
As well as in Theorem~\ref{t:BM},  
we need as a key assumption the convexity of $F$ in its second variable. However, for technical reasons which will be explained during the proof, 
here it is needed in the following stronger form:

\begin{itemize}
\item[(H3)'] {\it Reinforced convexity:}  $F$ is of class $C^2$, and for every $\delta >0$ there exists a positive constant $c_0$ such that
\[
\nabla_{X}^2 F (\xi, X) M \cdot M \geq c_0 |M| ^2 \qquad  \forall M, X \in S^n\,, \ \forall \xi \in \R^n \text{ with } |\xi| > \delta\,.
\]
\end{itemize}

\begin{theorem}[log-concavity of eigenfunctions]\label{t:log}
Assume that $F$ satisfies conditions (H1)-(H2)-(H3)'. 
Then:
\begin{itemize}
\item[(i)] if $\Omega$ is a strongly convex bounded open set of class $C^{2, \beta}$ for some $\beta\in (0,1)$, then any positive viscosity eigenfunction is log-concave;
\item[(ii)] 
if $\Omega$ is a convex bounded open set,
then there exists a positive viscosity eigenfunction which is log-concave. 
\end{itemize}
\end{theorem}

The above theorem can be read as an extension to viscosity solutions
of general fully  nonlinear operators 
of the result proved by Sakaguchi  in  \cite{Sak}  for the $p$-Laplacian  
(see also \cite{kuhn})
and by Bianchini and Salani in \cite{BS13}
for a general class of operators including the ones considered here.
Part (i) of the statement 
 is obtained essentially via the convex envelope method of Alvarez-Lasry-Lions, 
 whereas, for part (ii), we use our afore mentioned existence result (Theorem~\ref{t:approx}), which 
 involves an approximation argument  with smooth domains. In particular, the  fact that an approximation procedure is needed explains why part (ii) of the statement is formulated for \textit{some} (not for \textit{any}) positive viscosity eigenfunction. 
Clearly, in case the eigenvalue is  simple, also for 
$\Omega$ as in (ii)
any  positive viscosity solution is log-concave. This is for instance the case of 
the $p$-Laplacian \cite{Sak} and of the normalized $p$-Laplacian \cite{CFK}.

\smallskip
 We conclude this Introduction by providing a short list of some relevant operators to which the results stated above apply. 
\begin{example}  
The following operators satisfy assumptions (H1)-(H2)-(H3). 
Moreover, 
all of them satisfy also assumption (H3)' 
(the corresponding function $F$ being linear in $X$), 
except for the minimal Pucci operator, which however satisfies 
assumption (H3) (see \cite[Lemma 2.10]{CC}).

\begin{itemize}
\item{} {\it The $p$-Laplacian, for $p> 1$:}
\[
\begin{array}{ll}
& \Delta _p u = {\rm div} \big ( |\nabla u | ^ {p-2} \nabla u )
\\  \noalign{\bigskip}
&
F (\xi, X) = - |\xi| ^ {p-2} {\rm tr} X - (p-2) |\xi| ^ {p-4} \langle X \xi, \xi \rangle\, , \quad \alpha = p-2
\end{array}
\]

\medskip
\item{} {\it The normalized $p$-Laplacian,  for $p > 1$:}
\[
\begin{array}{ll}
& \Delta _p^N  u = \frac{1}{p} |\nabla u| ^ {2-p} {\rm div} \big ( |\nabla u | ^ {p-2} \nabla u )
\\  \noalign{\bigskip}
&
F (\xi, X) = -\frac{1}{p}  {\rm tr} X - \frac{p-2}{p} |\xi| ^ {-2} \langle X \xi, \xi \rangle \, , \quad \alpha = 0
\end{array}
\]

\item{} {\it The minimal Pucci operator:}
\begin{align*}
&  \mathcal M _{\lambda, \Lambda} ( D^ 2 u) = \lambda \sum _{e_i>0}  e_i (D^2 u) + \Lambda \sum _{e_i <0} e_i  (D^2 u)\,,
\quad
0 < \lambda \leq \Lambda,
\\
& 
F (\xi, X) =   \lambda \sum _{e_i>0}  e_i (X) + \Lambda \sum _{e_i <0} e_i  (X)\, , \quad \alpha = 0 \quad (e_i (X)= \text{eigenvalues of } X).
\end{align*}
\end{itemize}
\end{example}

The remaining of the paper is organized as follows:
\begin{itemize}
\item[--] in Section \ref{sec:prel} we provide the intermediate results we need about viscosity solutions  and infimal convolutions;

\item[--]  in Section \ref{ss:existence} we prove the existence of eigenfunctions for domains in $\mathcal A ^n$; 

\item[--] in Section \ref{sec:proofs} we give the
proofs of Theorems \ref{t:BM} and \ref{t:log}. 
\end{itemize}

\section{Preliminary results}\label{sec:prel}

\subsection{Viscosity solutions and maximum principle}\label{sec:lambda}

Below we adopt the following standard notation:
if $u,\varphi$ are two real functions on $\Omega$ and
$x\in\Omega$, by writing
$
\varphi \prec_x u
$ (resp. $u \prec _x \varphi$),
we mean that {\it $\varphi$ touches $u$ from below (resp.\ from above)  at $x$}, that is $u(x) = \varphi(x)$ and $\varphi(y) \leq u(y)$ (resp. $u (y) \leq \varphi (y)$) for every $y\in\Omega$.
Moreover, we denote by $\osubjet u (x)$ (resp.\ $\osuperjet u (x)$) the {\it second order sub-jet} (resp.\ {\it super-jet})
of $u$
at $x$,  which is by definition the set of pairs
$(\xi , A) \in \R ^n \times S^n$ such that, as $y \to x,\ y\in \overline{\Omega}$, it holds
$$
 u (y) \geq  (\leq) \ u ( x) + \pscal{ \xi}{y- x}
+ \frac{1}{2} \pscal {A (y- x)}{y- x} + o ( |y - x|^2)
\,.
$$

For any $\lambda>0$, the notion of viscosity sub- and super-solutions to  the pde 
\[
F(\nabla  u, D ^ 2 u ) = \lambda  |u|^{\alpha} u
\]
can be intended according  Crandall-Ishii-Lions \cite{CIL} or according to
 Birindelli-Demengel \cite{BiDe2007b}, as formulated respectively in Definition \ref{d:CIL} and Definition \ref{d:BD}.  
For later use, we give these two definitions for the more
general equation
\begin{equation}\label{f:pdeg}
 F(\nabla  u, D ^ 2 u ) = g(u),
\end{equation}
where $g\colon\R\to\R$ is a continuous function.

\begin{definition}\label{d:CIL}
\smallskip
-- An upper semicontinuous function $u :\Omega \to \R$ is  a {\it viscosity sub-solution} to \eqref{f:pdeg}
if, for every $x\in\Omega$ and for every smooth function $\varphi$  such that
$u \prec_x \varphi$, denoting by $F_*$ the lower semicontinuous envelope of $F$, it holds
\[
F_* (\nabla \varphi(x), D ^ 2 \varphi(x)) \leq 
g(\varphi(x))
\]
(or equivalently 
$F_* (\xi, A)  \leq g(u(x))$
for every $(\xi, A) \in \osuperjet u (x)$).

\smallskip
-- A lower semicontinuous function $u :\Omega \to \R$ is  a {\it viscosity super-solution} to \eqref{f:pdeg} 
if, for every $x\in\Omega$ and for every smooth function $\varphi$  such that
$\varphi \prec_x u$,  denoting by $F^*$ the upper semicontinuous envelope of $F$, it holds
\[
F^* (\nabla \varphi(x), D ^ 2 \varphi(x)) \geq 
g(\varphi(x))
\]
(or equivalently 
$F^* (\xi, A)  \geq g(u(x))$
for every $(\xi, A) \in \osubjet u (x)$).

\smallskip
-- A continuous function $u\colon\Omega \to \R$ is a  {\it viscosity solution} to \eqref{f:pdeg}
in $\Omega$  if it is both a viscosity super-solution and a viscosity sub-solution.
\end{definition}

\begin{definition}\label{d:BD}

-- An upper semicontinuous function $u\colon \Omega \to \R$ is a 
{\it viscosity sub-solution} to  \eqref{f:pdeg}
if, for every $x\in\Omega$:

$\cdot)$ either $u$ is equal to a constant $c$ on an open ball $B_r(x) \subset \Omega$
and
$ 0 \leq g(c)$;  

$\cdot)$ or
for every
smooth function $\varphi$  such that
$u \prec_x \varphi$ with $\nabla\varphi(x)\neq 0$, it holds
\[
F (\nabla \varphi(x), D ^ 2 \varphi(x)) \leq 
g(\varphi(x))
\]
(or equivalently 
$F (\xi, A)  \leq g(u(x))$
for every $(\xi, A) \in \osuperjet u (x)$ with $\xi \neq 0$).

\smallskip
-- A lower semicontinuous function $u\colon\Omega \to \R$ is a {\it viscosity super-solution} of
\eqref{f:pdeg} if, for every $x\in\Omega$:

$\cdot)$ either $u$ is equal to a constant $c$ on an open ball $B_r(x) \subset \Omega$
and 
$0 \geq g(c)$;  

$\cdot)$  or
for every smooth function $\varphi$  such that
$\varphi \prec_x u$ with $\nabla\varphi(x)\neq 0$, it holds
\[
F (\nabla \varphi(x), D ^ 2 \varphi(x)) \geq 
g(\varphi(x))
\]
(or equivalently 
$F (\xi, A)  \geq g(u(x))$
for every $(\xi, A) \in \osubjet u (x)$ with $\xi \neq 0$).

\smallskip
-- A continuous function $u: \Omega \to \R$ is a {\it viscosity solution} to \eqref{f:pdeg} in $\Omega$  if it is both a viscosity supersolution and a viscosity subsolution.
\end{definition}

\medskip
The following equivalence lemma is adapted from \cite[Lemma 2.1]{DFQ10} and \cite[Proposition~2.4]{AtRu18}, and will be very useful in the sequel
({\it cf.}\ Remark \ref{r:kuhn}). 
For this result and the subsequent Theorem~\ref{t:birdem},
the uniform ellipticity condition (H2) can be replaced by the
much weaker \textit{degenerate ellipticity condition}:
\begin{itemize}
\item[(H2)']
$F(\xi, X) \geq F(\xi,Y)$ for every $\xi\in\R^n\setminus\{0\}$
and for every $X,Y\in S^n$, $X\leq Y$.
\end{itemize}

\medskip
\begin{lemma}\label{l:equivalence}  
For any operator $F$ satisfying (H2)' and
\begin{equation}\label{f:pos-env}
F^*(0,0) = F_*(0,0) = 0,
\end{equation}
and any continuous function $g$, 
Definitions~\ref{d:CIL} and~\ref{d:BD} are equivalent.
\end{lemma}

\begin{proof}
Let us show the equivalence for super-solutions, the case of sub-solutions being analogous. 
Let $u$ be a super-solution according to Definition \ref{d:CIL}
and let $x\in\Omega$. 
To show that $u$ is a super-solution according to Definition \ref{d:BD}, we have just to consider  the case when $u$ is equal to a constant $c$ on a ball $B _r (x)$, and show that $0 \geq g (c)$.  
Let us fix an arbitrary point $y\in B_r (x)$, and let us consider the test function $\varphi (z) =c - |z-y| ^ q$, with $q >2$. We have that $\varphi$ touches $u$ from below at $y$, with   $\nabla \varphi (y) = 0$ and $D ^ 2 \varphi (y) = 0$.
Therefore, by assumption
\[
F^* (\nabla \varphi(y), D ^ 2 \varphi(y)) \geq 
g(\varphi(y))
\]
or equivalently, in view of \eqref{f:pos-env},
\[
0 = F^* (0, 0) \geq 
g(c)
\,.
\] 

Conversely, let $u$ be a super-solution 
according to Definition \ref{d:BD} and let $x\in\Omega$. 
To  show  that
$u$ is a super-solution according to Definition \ref{d:CIL},
we have to consider just the situation when $\varphi$ touches $u$ from below at $x$ with $\nabla \varphi (x) = 0$. 
We distinguish two cases.  First case:
$u$ is equal to a  constant $c$ on an open ball  $B_r(x) \subset \Omega$. Then it holds  $0 \geq g (c)$  
 (because $u$ is assumed to be a super-solution according to Definition \ref{d:BD}), and
 $D ^ 2 \varphi (x) \leq 0$ (because $\varphi$ is touching from below the locally constant function $u$).
Observe that, if $X\leq 0$, by the degenerate ellipticity assumption (H2)'
we have that $F(\xi, X+Y) \geq F(\xi, Y)$ for every $\xi\neq 0$ and
$Y\in S^n$, so that, from \eqref{f:pos-env},
\[
F^*(0, X) \geq 0,
\qquad \forall X\leq 0,
\]
hence we conclude that
\[
F ^* (0, D ^  2\varphi (x)) \geq 0 \geq 
g(c)\,.
\]
Second case: $u$ is not equal to a  constant on any open ball  
$B_r(x) \subset \Omega$.
Given $y \in B _\rho (0)$, with $\rho >0$ small enough, we consider the function
\[
\varphi _y ( z) = \varphi ( y + z) \qquad \forall z \in B _ r (x)\,.
\]
Since it is not restrictive to assume that $x$ is a strict minimum point
of $u-\varphi$ in $\overline{B}_r(x)$, for $|y|$ small enough
we have that 
$\varphi _y$ touches $u$ from below at some point $x_y \in B _r (x)$.
We claim that, with no loss of generality,  we may assume that  there exists a sequence $y _ k \to 0$ such that $\nabla \varphi _{y _k} ( x _{y _k}) \neq 0 $ for every $k$.
If this is the case, by testing the equation at $x_{y _k}$, we obtain
\[
F (\nabla \varphi _{y _k} ( x _{y _k}),  D ^ 2\varphi _{y _k} ( x _{y _k}))   \geq  
g(\varphi _{y _k} ( x _{y _k}  )) \,,
\]
which by passing to the limsup as $k \to + \infty$ yields
\[
F ^* (0, D ^  2\varphi (x))  \geq 
g(\varphi(x))\,.
\]
Finally, it remains to prove the claim.  
By making $r$ smaller if necessary, we can assume that $x$ is the unique critical point of $\varphi$ in $B _ r (x)$. 
(This is immediate if $D ^ 2 \varphi (x)$ is invertible, and such condition can always be assumed up to replacing $\varphi(z)$ by $\varphi _\varepsilon (z ) = \varphi (z) - \frac{\varepsilon}{2} (z-x) ^ t M (z-x)$, being $M$ a positive definite matrix in $S^n$ such that $D ^ 2 \varphi (x) - \varepsilon M$ is invertible for all $\varepsilon >0$.)
Then, arguing by contradiction, and exploiting the fact that $x$ is the unique critical point of $\varphi$ in $B _ r (x)$, one can show that, if the sequence $ y _k$ would  not exist, $u$ should be constant around $x$
 (see
 \cite[Lemma 2.1]{DFQ10} or \cite[Proposition 2.4]{AtRu18} for more details).
\end{proof}

\bigskip
We remark that assumption \eqref{f:pos-env} is
fulfilled by every operator $F$ satisfying the homogeneity condition (H1)
with $\alpha > -1$. 
Hence, in view of Lemma \ref{l:equivalence}, in the remaining of the paper we write the words  sub- and super-solutions
referring indistinctly to Definition~\ref{d:CIL} or~\ref{d:BD}.

\smallskip
The following maximum principle will be used as a  keystone in our proof of Theorem \ref{t:BM}:

\begin{theorem}\cite[Thm.~3.3]{BiDe2006}\label{t:birdem}
Let $\Omega\subset \R ^n$ be an open bounded set
and let $F$ satisfy assumptions (H1)-(H2)'.
Let $\tau < \eigen (\Omega)$, 
and let $u$ be a viscosity sub-solution to
\[
F (\nabla u, D ^ 2 u ) = \tau |u| ^ \alpha u
\qquad  
\text{ in } \Omega,
\]
satisfying $u\leq 0$ on $\partial\Omega$.
Then $u \leq 0$ in $\Omega$.
\end{theorem}

The idea to prove Theorem \ref{t:BM} is to construct a subsolution which,
if the inequality \eqref{f:BMo} would be false, would violate the maximum principle above.
To that aim we drive our attention to the operation of infimal convolution.

\subsection{Infimal convolutions}

For a fixed $k\in\N$, set
\begin{gather*}
\Lambda^+_k :=
\Big\{
\blambda = (t_0, \ldots, t_k) :\
t_i\in (0,1]\,, \
\sum_{i=0}^{k}t_i = 1
\Big\}\,,
\\ 
\mathcal O _k:=
\Big \{ (\Omega _0, \ldots, \Omega _k ) \ :\ \Omega _i \subset \R ^n  \text{ open bounded set } \Big \}\,.
\end{gather*}
Given $(\Omega _0, \dots, \Omega _k) \in \mathcal O _k$ and
$\blambda\in\Lambda^+_k$, we consider the convex Minkowski combination
\[
\Omega_{\blambda} := t_0\, \Omega_0
+\cdots + t_k\, \Omega_k
 = \Big \{ \sum _{i=0} ^k t _i x _i \ :\ x_i \in \Omega _i \Big \} \,.
\]
Notice that $\Omega_{\blambda}$ is an open set: namely,  
if $x = \sum _{i=0} ^k t _i x _i$, for any $j \in \{0, \dots, k\}$ it holds 
\begin{equation}\label{f:inclusion}
 B_\delta(x_j) \subset\Omega_j   \ \Longrightarrow \
B_{t_j\, \delta}(x) \subset\Omega_{\blambda} \,.
\end{equation}
 Let $v _i : \Omega _i \to \R$, $i=0,\ldots,k$, be given functions. We can think $v_i$ as defined on $\R ^n$, by extending them to $+ \infty$ outside $\Omega _i$.

 We call {\it weighted infimal convolution %$\infconv{f}$
of the functions $v_0, \ldots, v_k$}
(with weight $\blambda$)
the function defined on $\R ^n$ by
\[
\infconv{v}(x) :=
\inf\left\{
\sum_{i=0}^k t_i\, v_i(x_i):\
x_0,\ldots,x_k\in\R^n,\
x = \sum_{i=0}^k t_i\, x_i
\right\}\,,
\quad x\in\R^n\,.
\]

Clearly,  the weighted infimal convolution $\infconv{v}$ has finiteness domain 
$${\rm Dom} (\infconv{v}) = \Omega _{\blambda} \,.$$

We say that the infimal convolution $\infconv{v}$ is {\it exact} at a point $x \in \Omega _{\blambda}$, if the above infimum is attained.

 \smallskip
 The next result is inspired from \cite[Propositions~1 and 4]{ALL}. Given a family of continuous functions bounded from below,  it provides a key information on the subjets of their weighted infimal convolution, provided the latter is exact.
\begin{proposition}
\label{p:ALL0}
Let $(\Omega _0, \dots, \Omega _k) \in \mathcal O _k$ and
$\blambda = (t_0, \ldots,t_k)\in \Lambda^+ _k$.
Let $v _i : \Omega_i \to \R$ be continuous functions bounded from below, and
assume that $\infconv{v}$ is exact at
$x\in\Omega_{\blambda}$, with
\begin{equation}\label{f:att}
\infconv{v} (x) = \sum_{i=0}^k t_i\, v_i(x_i),
\quad
x = \sum_{i=0}^k t_i\, x_i,
\
x_i\in\Omega_i\ \forall i = 0, \ldots, k\,.
\end{equation}
Then, for a given pair $(\xi, A) \in J^{2-} \infconv{v} (x)$,
and for every $\varepsilon > 0$,
there exist
$A_0, \ldots, A_k \in S^n$ such that
$(\xi, A_i) \in \overline{J}^{2-} v_i (x_i)$,
$i=0,\ldots, k$, and
\begin{equation}\label{f:matineq0}
A - \varepsilon A^2 \leq
\sum_{i=0}^k t_i A_i\,.
\end{equation}
If, in addition, $A\geq 0$, and $\varepsilon$ is small enough,
then $A_i\geq 0$ for every $i=0,\ldots,k$ and
\begin{equation}\label{f:matineq}
A - \varepsilon A^2 \leq
\left(
t_0\, A_0^{-1} + \ldots +
t_k\, A_k^{-1}
\right)^{-1}\,.
\end{equation}
\end{proposition}

\begin{remark}
The above result (and its proof) is quite similar to Proposition~1 in  \cite{ALL}. 
For completeness, we give the proof in some detail,
since we are going to exploit inequality~\eqref{f:matineq0}, which is 
not explicitly given in \cite{ALL}.
 \end{remark}

\begin{proof}
To simplify the notation, let us denote $w := \infconv{v}$.
Let $\varphi\in C^2(\Omega_{\blambda})$ be a test function
such that $\varphi \prec_x w$.
Let $(y_0, \ldots, y_k) \in \Omega_0 \times \cdots \times \Omega_k$.
By the definition of $w$, the fact that
$(w-\varphi)(y) \geq (w-\varphi)(x)$ for every
$y\in\Omega_{\blambda}$, and
since $\infconv{v}$ is exact at $x$,
we have that
\[
\begin{split}
\sum_{i=0}^k t_i v_i(y_i) - \varphi\left(\sum_{i=0}^k t_i \, y_i\right)
& \geq
w\left(\sum_{i=0}^k t_i \, y_i\right) -
\varphi\left(\sum_{i=0}^k t_i \, y_i\right)
\\ & \geq
w\left(\sum_{i=0}^k t_i \, x_i\right) -
\varphi\left(\sum_{i=0}^k t_i \, x_i\right)
\\ & =
 \sum_{i=0}^k t_i\, v_i(x_i)
-\varphi\left(\sum_{i=0}^k t_i \, x_i\right)
\,.
\end{split}
\]
In other words, the point $(x_0, \ldots, x_k)$ where the infimum in \eqref{f:att} is attained turns out to be a minimum point
for the function
\[
\Omega_0\times\cdots\times\Omega_k \ni (y_0, \ldots, y_k) \mapsto
\sum_{i=0}^k t_i v_i(y_i) - \varphi\left(\sum_{i=0}^k t_i \, y_i\right)\,.
\]
Then, by \cite[Theorem~3.2]{CIL},
for every $\varepsilon>0$ there exist
$A_0, \ldots, A_k \in S^n$ such that
$(\xi, A_i) \in \overline{J}^{2-} v_i (x_i)$,
$i=0,\ldots, k$, and
\begin{equation}\label{f:matrix0}
\begin{pmatrix}
t_0\, A_0 & \cdots & 0 \\
\vdots & \ddots & \vdots \\
0 & \cdots & t_k\, A_k
\end{pmatrix}
\geq
\begin{pmatrix}
t_0^2\, B & \cdots & t_0t_k\, B \\
\vdots & \ddots & \vdots \\
t_0t_k\, B & \cdots & t_k^2\, B^2
\end{pmatrix}
\end{equation}
with $B := A - \varepsilon\, A^2$.

The inequality in \eqref{f:matineq0} follows by
testing \eqref{f:matrix0} with a vector of the form
$(h, \dots, h ) \in (\R ^n ) ^k$.

Moreover, by testing \eqref{f:matrix0} with vectors of the form $(0, \ldots, h_i, \ldots, 0)$,
we get  the inequalities
\begin{equation}\label{f:inoltre}
t_i ( A - \varepsilon A^2) \leq
 A_i,
\qquad \forall i=0,\ldots,k\,,
\end{equation}
 whereas, testing
\eqref{f:matrix0} with an arbitrary vector
$(h_0, \ldots, h_k)$, we see that
\begin{equation}\label{f:tomin}
\pscal{B\,h}{h}
\leq
\sum_{i=0}^k t_i\, \pscal{A_i\, h_i}{h_i}\,,
\qquad
\text{with}\
h := \sum_{i=0}^k t_i\, h_i\,.
\end{equation}

Assume now that $A\geq 0$,
and choose  $\varepsilon>0$ so that $I> \varepsilon A $,
and hence $B\geq 0$.
From \eqref{f:inoltre}, we see that $A_i\geq 0$ for every $i$.
In fact, it is not restrictive to assume that
$A_i$ are positive definite,
since the case of degenerate matrices can be handled
as in \cite{ALL}, p.~273.

Finally, minimizing the right-hand side of \eqref{f:tomin}
under the constraint $\sum_{i=0}^k t_i\, h_i = h$
leads to \eqref{f:matineq}.
\end{proof}

 \bigskip
 In order to be able to apply Proposition \ref{p:ALL0}, we complement it with the following statement, which provides sufficient conditions for the weighted convolution to be exact.

 \begin{proposition}\label{p:bordo}
 Let $(\Omega _0, \dots, \Omega _k) \in \mathcal O _k$ and
$\blambda = (t_0, \ldots,t_k)\in \Lambda^+ _k$.
Let $v _i : \Omega_i \to \R$ be continuous functions bounded from below, with
\begin{equation}\label{f:divi} v _i \to + \infty \quad \text{ as } x \to \partial \Omega _i \, , \quad \forall i = 1\dots, k\,.
\end{equation}
Then the weighted infimal convolution $\infconv{v}$ is continuous and exact at every point $x \in \Omega _{\blambda}$.
Moreover, it holds
\begin{equation}\label{f:div} \infconv{v} \to + \infty \quad \text{ as } x \to \partial \Omega _{\blambda}\,.
\end{equation}
\end{proposition}

\begin{proof}
 For the continuity of the weighted infimal convolution and the fact that it is exact, we refer to  \cite{Strom1996}, Theorem~2.5 and Corollary~2.1. In order to prove the last part of the statement, let us consider a sequence of points $x ^n \to \partial \Omega _{\blambda}$ as $n \to + \infty$. Since the   weighted infimal convolution is exact, there exists sequences $x_i ^n$, $i = 0, \dots, k$,  such that
$$ \infconv{v} (x^n) = \sum_{i=0}^k t_i\, v_i(x^n_i),
\quad
x^n = \sum_{i=0}^k t_i\, x^n_i,
\quad
x^n_i\in\Omega_i\ \  \forall i = 0, \ldots, k\,.
$$
 We claim that $x ^n _i \to \partial \Omega _i$ as $n \to + \infty$, $\forall i = 1, \dots, k$.
 Once proved the claim, the required property \eqref{f:div} follows at once from \eqref{f:divi}
 and the assumptions that the functions $v_i$'s are bounded from below.

To show the claim  it is enough to observe that, if $\dist(x, \partial\Omega_{\blambda}) < \delta$,
then $\dist(x_i, \partial\Omega_i) < \delta / t_i$
for every $i=0,\ldots,k$.
Indeed, if we assume by contradiction that there exists $j$ such that
$\dist(x_j, \partial\Omega_j) \geq \delta / t_j$, then
$B_{t_j^{-1}\, \delta}(x_j) \subset \Omega_j$.
By \eqref{f:inclusion}, this implies
$B_\delta(x) \subset\Omega_{\blambda}$,
contradiction.
\end{proof}

\subsection{The modified equation}
In view of Proposition \ref{p:bordo}, it is convenient to look at the equation satisfied by minus the logarithm of viscosity eigenfunctions, so to deal with functions which diverge on the boundary. To that aim,
let us introduce the operator $G$ associated  with $F$ by
\begin{equation}\label{f:H}
G(\xi, X) := -F(\xi, \xi\otimes\xi - X)\,,
\end{equation}
 and let us consider the modified equation
\begin{equation}\label{f:eqmod1}
G (\nabla v, D ^ 2 v ) = - \eigen(\Omega) \quad \text{ in } \Omega \, .
\end{equation}

 \begin{remark}\label{r:G}
If $F$ satisfies (H2) (resp.\ (H2)'), the same holds for $G$. Moreover, if $F$ satisfies (H3), namely $F$ is convex in $X$, then $G$ is concave in $X$. \end{remark}

\begin{remark}\label{r:eqmod}
Similarly as in Section \ref{sec:prel}, also viscosity sub- and super-solutions to \eqref{f:eqmod1} can be intended either \`a la Crandall-Ishii-Lions or \`a la Birindelli-Demengel, namely according to Definition \ref{d:CIL} or to Definition \ref{d:BD}. Thanks to Lemma \ref{l:equivalence}, 
the two notions are equivalent. Note in particular that, since the right--hand side of \eqref{f:eqmod1} is negative, for super-solutions the ``either'' condition in Definition \ref{d:BD}  is automatically satisfied.   
\end{remark}

\begin{lemma}\label{l:log} 
Assume that $F$ satisfies (H1)-(H2)', and let $G$ be defined by \eqref{f:H}.
Then a function $u$  is a positive   viscosity sub-solution 
to
\[
 F(\nabla  u, D ^ 2 u ) = \eigen (\Omega)  u ^ {\alpha + 1}  \quad \text{ in } \Omega
\]
if and only if the function $v = - \log u$ is a viscosity super-solution to 
\[
G (\nabla v, D ^ 2 v ) = - \eigen(\Omega) \quad \text{ in } \Omega \, .
\]
\end{lemma}

\begin{proof}
Let us give the proof working with solutions \`a la Crandall-Ishii-Lions. 
We observe that $u \prec_x \varphi$ if and only if $\psi:= - \log \varphi \prec _x v  $, and that the inequality $F_*(\nabla \varphi, D ^ 2 \varphi)  \leq \eigen(\Omega) \varphi ^ {\alpha +1} $
can be rewritten as
$$F_* \big ( - e ^ { - \psi} \nabla \psi, e ^ { - \psi} (\nabla \psi \otimes \nabla \psi - D ^ 2 \psi)\big ) \leq
\eigen (\Omega)    e ^ { - (\alpha + 1)\psi}\,. $$
By (H1), this amounts to
$$F_* \big (   \nabla \psi,  \nabla \psi \otimes \nabla \psi - D ^ 2 \psi \big ) \leq \eigen (\Omega)    \,.  $$
The required equivalence follows by observing that 
\[
G^*(\xi, X) = [-F(\xi, \xi\otimes\xi - X)]^*
= - F_*(\xi, \xi\otimes\xi - X)\,.
\qedhere
\]
\end{proof}

\bigskip

We are finally in a position to give the main brick for the proof of Theorem \ref{t:BM}:

\begin{proposition}\label{p:concave}
Assume that $F$ satisfies 
$F_*(0,0) = 0$, (H2)' and (H3),  and let $G$ be defined by \eqref{f:H}.  
Let $(\Omega _0, \dots, \Omega _k) \in \mathcal O _k$ 
, and
$\blambda = (t_0, \ldots,t_k)\in \Lambda^+ _k$.
Let $v_0, \ldots, v_k$ be continuous functions bounded from below which are viscosity super-solutions to
\[
\begin{cases}
G(\nabla v_i, D^2 v_i) = - \eigen (\Omega _i)  & \text{in}\ \Omega_i,\\
v_i \to +\infty &\text{on}\ \partial\Omega_i\,.
\end{cases}
\]
Then $w = \infconv{v}$ is a viscosity super-solution to
\[
\begin{cases}
G(\nabla w, D^2 w) = - \sum _{i=0} ^k t_i \eigen (\Omega _i)  & \text{in}\ \Omega_{\blambda},\\
w \to +\infty &\text{on}\ \partial\Omega_{\blambda}\,.
\end{cases}
\]
\end{proposition}

\begin{proof} 
From Proposition \ref{p:bordo}, we know that $w$ is continuous, exact, and satisfies $w \to + \infty $ as  $x \to \partial \Omega _{\blambda}$.  In order to check that $w$ is a viscosity super-solution to 
$G(\nabla w, D^2 w) = - \sum _{i=0} ^k t_i \eigen (\Omega _i) $ in $\Omega_{\blambda}$, we use the definition \`a la Birindelli-Demengel.  
Let $x\in\Omega_{\blambda}$. If $w$ is constant on a ball centered at $x$, we have nothing to check.
Otherwise,  let $(\xi, A) \in J^{2-} w(x)$, with $\xi \neq 0$. Let $x _i \in \Omega _i$ be such that $\eqref{f:att}$ holds.
By Proposition~\ref{p:ALL0}, there exist
$A_0, \ldots, A_k\in S^n$ such that
$(\xi, A_i) \in \overline{J}^{2-} v_i(x_i)$,
satisfying \eqref{f:matineq0}.
Hence,
\[
\begin{split}
G(\xi, A - \varepsilon\, A^2)  \geq
G\left(\xi, \sum_{i=0}^k t_i A_i\right)
\geq
\sum_{i=0}^k t_i G (\xi, A_i)
\geq - \sum _{i=0} ^k t_i \eigen (\Omega _i) \, ,
\end{split}
\]
where in the first inequality we have used the fact that $G$ is degenerate elliptic, 
in the second one the fact that it is concave in $X$  ({\it cf.} Remark~\ref{r:G}), and in the third one the fact that the $v_i$'s are super-solutions to 
$G(\nabla v_i, D^2 v_i) = - \eigen (\Omega _i)$. 

Passing to the limit as $\varepsilon\to 0$ we conclude that
$G (\xi, A) \geq - \sum _{i=0} ^k t_i \eigen (\Omega _i) $.
\end{proof}

\bigskip
\begin{remark}\label{r:kuhn} We warn the reader that the above proof cannot be successfully concluded if one 
adopts the definition of viscosity super-solution \`a la Crandall-Ishii-Lions. Indeed, in this case, one would need to use the concavity
of the upper semicontinuous envelope $G^*$. But, in general, 
the concavity of $G$ is not inherited by $G ^*$  (for instance, 
in case of the normalized $p$-Laplacian, 
one can easily check that $G^*$ fails to be concave). This sheds some light on the importance of the equivalence Lemma \ref{l:equivalence}. 
\end{remark}

\section{Existence of eigenfunctions for domains in $\clset$}
\label{ss:existence}

In this section we prove the existence of eigenfunctions
for operators $F$ satisfying assumptions (H1)-(H2) on domains belonging to the class $\clset$
defined in \eqref{f:sets} (see Theorem~\ref{t:approx}),
along with their global H\"older continuity (see Proposition~\ref{p:holder}).
We remark that the restriction $\alpha > -1$ in (H1) is fundamental
for the proof of Lemma~\ref{l:barr} below, and hence also for 
the subsequent results.
For domains of class $C^2$,
the corresponding results have been proved
in \cite[Theorem~5.5 and~4.1]{BiDe2006} 
(see also \cite[Theorem~8 and Proposition~6]{BiDe2007b}).

\smallskip

We recall that, for any Lipschitz domain $\Omega$,  denoting by
$d _\Omega$ the distance function from the boundary
\[
d_\Omega (x) := \min_{y\in\partial\Omega} |y-x|,
\qquad x\in\R^n,
\]
the following properties are equivalent
(see e.g.\ \cite{CSW,CoTh,CFb}):
\begin{itemize}
\item[(a)] $\Omega\in\clset$;
\item[(b)] 
there exists $r>0$ such that
the distance function $d_{{\Omega}}$ is differentiable
at any point of the exterior tubular neighborhood
\[
\mathcal{N}_r := \{x\in\R^n \setminus \Omega :\ 0 < d_{{\Omega}}(x) < r\};
\]
\item[(c)] $\Omega$ is a set of \textsl{positive reach},
i.e.\ there exists $r>0$ such that
every point $x\in \mathcal{N}_r$ admits a unique projection on
$\overline{\Omega}$.
\end{itemize}

These properties are clearly satisfied if $\Omega$ is of class $C^2$ or
if $\Omega$ is a convex set.

\smallskip
Let us also recall that, 
if  $\Omega\in\clset$, the distance function $d_\Omega$ 
is semiconcave in $\overline{\Omega}$, i.e.\
there exists a constant $\kappa>0$ such that
the map 
$x\mapsto d_\Omega(x) - \frac{\kappa}{2}|x|^2$ 
is concave in $\overline{\Omega}$ (see \cite[Proposition~2.2(iii)]{CaSi}). 
The constant $\kappa$ is called a semiconcavity constant for $d _\Omega$,
and can be chosen equal to the reciprocal of the radius in
the uniform external sphere condition. 

As a consequence of  the  semiconcavity of $d _\Omega$, for any $\Omega$  in $\clset$  and any function $f$ bounded in $\overline \Omega$, we are able to construct a barrier for sub-solutions to  
\begin{equation}
\label{f:equaf}
\begin{cases}
F(\nabla u, D^2 u) = f(x),
&\text{in}\ \Omega,
\\
u =0
&\text{on}\ \partial\Omega.
\end{cases}
\end{equation}

We prove: 
\begin{lemma}
\label{l:barr}
Let $\Omega\in\clset$,
let $F$ satisfy (H1)-(H2), and let
$f$ be a bounded function
in $\overline{\Omega}$.
Then, for every upper semicontinuous sub-solution $u$ of \eqref{f:equaf}
and every $\gamma\in(0,1)$,
there exist constants $M, \delta > 0$, depending only on the semiconcavity constant of $d _\Omega$  and on the structural constants of $F$,  
such that
\[
u(x) \leq M\, d _\Omega (x)^\gamma,
\qquad
\forall x\in\overline{\Omega}\ \text{such that}\ d_\Omega (x) \leq \delta.
\]
\end{lemma}

\begin{proof} 
Throughout the proof we write for brevity $d$ in place of $d_\Omega$. 
If $\kappa >0 $ is a semiconcavity constant for $d$, since the map 
$x\mapsto d(x) - \frac{\kappa}{2}|x|^2$ 
is concave in $\overline{\Omega}$, we have
\begin{equation}
\label{f:subd}
x\in\Omega, \quad
(\xi, A) \in J^{2,-} d(x)
\qquad\Longrightarrow\qquad
\nabla d (x) = \xi,
\
A\leq \kappa\, I.
\end{equation}

Let $\gamma\in (0,1)$ be fixed, and let us consider the function
$g(x) := M\, d(x)^\gamma$, where $M>0$ is a constant that will be 
determined later.
For every $x\in\Omega$, by \eqref{f:subd} we have that
$(\xi, A)\in J^{2,-} d(x)$ if and only if
$(\zeta, X)\in J^{2,-} g(x)$, with
\[
\zeta = M\gamma d(x)^{\gamma-1} \xi,
\quad
X = M\gamma d(x)^{\gamma -2} 
\left[(d(x)\, A + (\gamma-1)\xi \otimes \xi\right],
\quad A\leq \kappa\, I,
\quad |\xi| = 1,
\]
and, in this case, both $d$ and $g$ are differentiable at $x$,
with $\nabla d(x) = \xi$ and $\nabla g(x) = \zeta\neq 0$.

Hence, if $x\in\Omega$ %with $d(x) \leq \delta$ 
and $(\zeta, X)\in J^{2,-} g(x)$,
from (H1)-(H2) it holds
\[
\begin{split}
F(\zeta, X) & =
(M\,\gamma)^{\alpha + 1} d(x)^{(\alpha+1)\gamma-\alpha-2}\,
F(\xi, d(x) A + (\gamma-1) \xi\otimes \xi)
\\ & \geq
(M\,\gamma)^{\alpha + 1} d(x)^{(\alpha+1)\gamma-\alpha-2}\,
F(\xi, \kappa\, d(x)\, I - (1-\gamma) \xi\otimes \xi)
\\ & \geq
(M\,\gamma)^{\alpha + 1} d(x)^{(\alpha+1)\gamma-\alpha-2}\,
\left[
c(1-\gamma) - C\,n\,\kappa\,d(x)
\right]\,,
\end{split}
\]
where in the last inequality we have used the fact that $|\xi|=1$.

Since the exponent $[(\alpha+1)\gamma-\alpha-2]$ is negative,  if we choose $\delta < c(1-\gamma) / (C\,n\,\kappa)$, 
we conclude that there exists $\varepsilon > 0$,
depending only on $\gamma$ and $\kappa$
(and on the structural constants of $F$),
such that
\[
F(\zeta, X) \geq M^{\alpha+1}\varepsilon,
\qquad
\forall\ (\zeta, X)\in J^{2,-} g(x),
\ \text{with}\ x\in\Omega,\ d(x)\leq\delta.
\]
In other words, $g$ is a positive supersolution of the equation
\[
F(\nabla g, D^2 g) \geq M^{\alpha+1}\varepsilon
\qquad
\text{in}\
\Omega_{\delta} := \{x\in\Omega:\ d(x) < \delta\}.
\]
Finally, we can now choose
\[
M := \max\left\{
\delta^{-\gamma}\max_{x\in\overline{\Omega}_\delta} u \, ,
\, \left(\frac{\|f\|_\infty}{\varepsilon}\right)^{\frac{1}{\alpha+1}} + 1
\right\}\,,
\]
so that $g \geq u$ on $\partial\Omega_\delta$
and $|f(x)| < M^{\alpha+1}\varepsilon$ for every $x\in\overline{\Omega}_\delta$,
hence the claim follows from the comparison result proved in
\cite[Theorem~3.6]{BiDe2006}.
\end{proof}

\bigskip
We can now derive a global H\"older estimate:

\begin{proposition}
\label{p:holder}
Let $\Omega\in\clset$,
let $F$ satisfy (H1)-(H2), let $f$ be a bounded function
in $\overline{\Omega}$,
and let $u$ be a non-negative bounded viscosity solution of~\eqref{f:equaf}.

Then, for every $\gamma\in (0,1)$ there exists a constant
$H>0$, depending only on $\gamma$, $\|f\|_\infty$
and the semiconcavity constant of $d_\Omega$,
such that
\[
|u(x) - u(y)|
\leq H \, |x-y|^\gamma,
\qquad
\forall x,y\in\overline{\Omega}.
\]
\end{proposition}

\begin{proof} 
Thanks to Lemma~\ref{l:barr},
the result can be obtained following line by line
the proof of Proposition~6 in \cite{BiDe2007b}
(see also \cite[Theorem~4.1]{BiDe2006}).
\end{proof}

\begin{remark}
As a consequence of the global H\"older estimate given in Proposition \ref{p:holder}, 
it is possible to obtain also 
a local Lipschitz regularity result. More precisely, 
under the hypotheses of Proposition~\ref{p:holder},
assume in addition that $F$ sastisfies the following H\"older continuity assumption with respect to $\xi\neq 0$:
there exist $\mu\in \left(1/2, 1\right]$ and $K>0$ such that
\[
|F(\xi+\zeta, X) - F(\xi, X)| \leq
K |\zeta|^\mu |X|,
\qquad
\forall\
|\xi| = 1,\ |\zeta| < 1/2,\ X\in S^n\,.
\]
Then, by arguing as in Theorem~4.2 of \cite{BiDe2006}, one can see that every non-negative bounded viscosity solution of~\eqref{f:equaf}
is locally Lipschitz continuous in $\Omega$.
\end{remark}

\bigskip

Finally, thanks to Proposition \ref{p:holder} we are in a position to give

\begin{theorem}\label{t:approx}
Let $\Omega\in\clset$
and let $F$ satisfy (H1)-(H2).
Then for $\lambda = \eigen(\Omega)$ the eigenvalue problem \eqref{f:pb}
admits a positive viscosity solution $u$.
Moreover, $u$  can be obtained as the uniform limit
of a sequence of positive eigenfunctions $\{u _k\}$, associated with an increasing sequence 
of smooth domains $\{\Omega_k\}$
such that
\[
\bigcup_k \Omega_k = \Omega,
\qquad
\lim_{k\to + \infty} \eigen(\Omega_k) = \eigen(\Omega)\,.
\]
\end{theorem}

\begin{proof}
Since $\Omega$ satisfies a uniform exterior sphere condition, we can construct a sequence of smooth ($C^\infty$) domains
$\{\Omega_k\}$,
still satisfying a uniform sphere condition 
(possibly with a smaller radius $r$, independent of $k$),
such that
$\overline{\Omega}_k\subset \Omega_{k+1}$ and
$\bigcup\Omega_k = \Omega$.
This can be achieved by a standard regularization argument, 
i.e.\ by mollifying the function whose graph locally defines the boundary of $\Omega$. 

For every $k\in\N$, let  now $u_k$ be a positive eigenfunction in $\Omega_k$,
normalized by $\|u_k\|_\infty = 1$, and let us extend it in $\overline{\Omega}$
by setting $u_k = 0$ in $\overline{\Omega}\setminus\Omega_k$.

Let us fix $\gamma\in (0,1)$.
By Proposition~\ref{p:holder}, there exists a constant
$H > 0$, depending only on $r$, such that
\[
|u_k(x) - u_k(y)| \leq H |x-y|^\gamma,
\qquad
\forall x,y\in\overline{\Omega},
\quad
\forall k\in\N.
\]
Hence, by the Ascoli--Arzelà theorem, from $\{u_k\}$
we can extract a subsequence that converges uniformly in
$\overline{\Omega}$ to some continuous function $u$. Moreover, by monotonicity, the sequence $\eigen (\Omega _k)$ converges decreasingly to some limit $L$.  
Thus the function  $u$ is a non-negative viscosity solution to the equation  
$F (\nabla u, D ^ 2 u ) = L |u | ^ \alpha u$ in $\Omega$. 
Since $u\geq 0$ and $u\not\equiv 0$, 
by the strict maximum principle proved in \cite[Theorem~2]{BiDe2007b}
we deduce that $u>0$ in $\Omega$.
By definition of $\eigen (\Omega)$, this gives the inequality  
$\eigen (\Omega ) \geq L$. 
On the other hand, since $\Omega _k \subset \Omega$, we have $\eigen (\Omega ) \leq \eigen (\Omega _k)$ and hence in the limit $\eigen (\Omega) \leq L$, 
so that $u$ is a positive eigenfunction
associated with $\eigen(\Omega)$.
\end{proof}

\section{Proofs of Theorems \ref{t:BM} and \ref{t:log}}
\label{sec:proofs}

\subsection{Proof of Theorem \ref{t:BM}}
\label{ss:BM}
First of all we observe that it is enough to prove the inequality
\begin{equation}\label{f:BM}
\eigen \big ((1-t) \Omega _0 + t \Omega _1 \big )
\leq
(1-t)\eigen(\Omega_0) + t\, \eigen(\Omega_1),
\qquad
\forall t\in [0,1].
\end{equation}
Indeed, by a standard argument,
the Brunn--Minkowski inequality \eqref{f:BMo} follows from \eqref{f:BM}
and the fact that
\[
\eigen(k\, \Omega) = \frac{1}{k^{\alpha+2}}\, \eigen(\Omega),
\qquad \forall k>0.
\]
Namely, it is enough to apply \eqref{f:BM} with
\[
t' = \frac{t\,\eigen(\Omega_1)^{-1/(\alpha+2)}}{(1-t)\, \eigen(\Omega_0)^{-1/(\alpha+2)} + t\, \eigen(\Omega_1)^{-1/(\alpha+2)}}\,,
\qquad
\Omega_i' = \eigen(\Omega_i)^{1/(\alpha+2)}\, \Omega_i,
\ i = 0,1.
\]

Let us prove \eqref{f:BM}.
For $i=0,1$, thanks to Theorem \ref{t:approx}, there exists a positive eigenfunction  $u_i$ 
associated with $\eigen(\Omega_i)$, i.e.\
a positive function in
$C(\overline{\Omega}_i)$ which is a viscosity solution to
\[
\begin{cases}
F(\nabla u_i, D^2 u_i) = \eigen(\Omega_i)\, u_i^{\alpha+1}
&\text{in}\ \Omega_i,\\
u_i = 0
&\text{on}\ \partial\Omega_i.
\end{cases}
\]

By Lemma \ref{l:log}, for $i=0,1$, the function $v_i := -\log u_i$
is a viscosity super-solution of
$$\begin{cases}
G(\nabla v_i, D^2 v_i) = -\eigen(\Omega_i)
&\text{in}\ \Omega_i,\\
v_i \to +\infty
&\text{on}\ \partial\Omega_i,
\end{cases}
$$
where $G$ is the function defined in \eqref{f:H}.

Let $w\colon\Omega_{\blambda}\to\R$ be the infimal convolution
of $v_0, v_1$ with coefficients $\blambda=  (1-t, t)$,  in $\Omega_{\blambda} = (1-t) \Omega_0 + t \Omega_1$, {\it i.e.},
\[
w(x) :=
\inf \big \{
(1-t)v_0(x_0) + t\, v_1(x_1):\
x_0\in\Omega_0,\
x_1\in\Omega_1,\
x = (1-t)x_0 + t\, x_1
\big \}\,.
\]
By Proposition \ref{p:concave}, $w$ is a  viscosity super-solution to
\[
\begin{cases}
G(\nabla w, D^2 w) = - \big [(1-t)\eigen(\Omega_0) + t\, \eigen(\Omega_1)\big ]  & \text{in}\ \Omega_{\blambda},\\
w \to +\infty &\text{on}\ \partial\Omega_{\blambda}\,.
\end{cases}
\]

Let us define
$\ubar\colon\overline{\Omega}_{\blambda}\to \R$ as
$\ubar(x) := e^{-w(x)}$ for every $x\in\Omega_{\blambda}$,
$\ubar(x) = 0$ for every $x\in\partial\Omega_{\blambda}$.
Clearly, $\ubar > 0$ in $\Omega_{\blambda}$
and $\ubar \in C(\overline{\Omega}_{\blambda})$.

Moreover, applying again Lemma \ref{l:log}, we infer that
$\ubar$ is a viscosity sub-solution to
\[
F(\nabla \ubar, D^2 \ubar) = [(1-t)\eigen(\Omega_0) + t\, \eigen(\Omega_1)]\,\ubar^{\alpha+1}
\qquad
\text{in}\ \Omega_{\blambda}.
\]

\smallskip
Since $\ubar > 0$ in $\Omega_{\blambda}$ and $\ubar = 0$ on $\partial\Omega_{\blambda}$,
by Theorem \ref{t:birdem} we conclude that
\eqref{f:BM} holds.
\qed

\subsection{Proof of Theorem \ref{t:log}(i)}
Let $v := - \log u$. In order to prove that $v$ is a convex function, we exploit the convex envelope method by Alvarez-Lasry-Lions. 
By definition, the convex envelope $v_{**}$ of $v$ satisfies $v_{**} \leq v$.
In order to show the converse inequality, we apply a comparison argument to the modified equation
\begin{equation}\label{f:eqmod}
G (\nabla v, D ^ 2 v ) = - \eigen(\Omega)
\end{equation}
settled on a suitable  level set $\Omega _\e := \{ u > \e  \}$.
To be more precise, 
the comparison principle given in \cite[Theorem 1.3]{LuWang2008} ensures that the inequality $v_{**} \geq v$ in $\Omega _\e$ holds true in $\Omega _\e$
(and hence in the limit as $\e \to 0 ^+$ also in $\Omega)$, provided the following two properties hold true:
\begin{itemize}
\item[(a)] $v_{**}$ is a viscosity super-solution to  \eqref{f:eqmod} in $\Omega _\e$;
 \item[(b)]  $v_{**} = v$ on $\partial \Omega _\e$.
\end{itemize}
We point out that we cannot take $\e = 0$ (namely work directly on $\Omega$) because $v \to + \infty$ on $\partial \Omega$.
We also stress that the assumption (H3)' intervenes in the proof of item~(b) given below, and this is the reason why the statement cannot be proved under the
weaker condition $X \mapsto F (\xi, X ^ { -1} )$ convex appearing in \cite{ALL}. 

\medskip
{\it Proof of (a).} Let us show that $v_{**}$ is a viscosity super-solution to \eqref{f:eqmod} in the whole $\Omega$.
  We observe that
\[
v_{**}  = \min  \Big \{ (v \sharp \cdots \sharp v)_{\blambda}  \ :\ \blambda \in \!\!\!\!\! \bigcup _{k \leq (n+1)}\!\!\!\!\!  \Lambda _k ^+ \Big \} \,.
\]
Thus, for some $t \in \Lambda_k ^+$ (depending on $x$), we have
\[
v_{**} (x) =  (v \sharp \cdots \sharp v)_{\blambda} (x) ,
\]
and hence, by applying Proposition \ref{p:concave} (with $\Omega _i = \Omega$ and $v_i = v$ for every $i$), we conclude that $v_{**}$ is a super-solution to \eqref{f:eqmod}. (As well, one could apply here Proposition 3 in \cite{ALL}). 

\medskip
{\it Proof of (b).}  
By Lemma 4 in \cite{ALL}, the required equality $v_{**} = v$ on 
$\partial \Omega _\e$ is satisfied provided the level set $\Omega _\e$ is
{\it convex}. We are thus reduced to prove the convexity of $\Omega_\e$ for $\e$ small enough.

We start by noticing that, by \cite[Proposition 3.5]{BirDem2010}, $v$ belongs to $C ^ {1, \beta} (\overline \Omega)$ (for some $\beta \in (0, 1)$). Combined with the Hopf boundary point principle given in \cite[Corollary 1]{BiDe2007b}, this ensures that 
 $|\nabla v | \geq \alpha >0$ in $\overline N$, where $N$ is an inner neighbourhood of $\partial \Omega$. 
This fact, and the strong convexity assumption made on $\Omega$, enable us to apply Lemma 2.4 in \cite{Kor}  (see also \cite[Proposition 3.2]{Sak}) to infer that the required convexity property of $\Omega _\e$ is satisfied, for $\e$ sufficiently small, as soon as we know that $u \in C ^ 2 (\overline {N})$.

The latter property follows by standard elliptic regularity, in particular
thanks to the convexity hypothesis made on $F$ and to the condition $\partial \Omega \in C ^ {2, \beta}$. So  we limit ourselves to give adequate references, along with a few additional comments.
By the convexity of $F$, we can apply the method of continuity  as done for instance in the proof of Theorem 9.7 in \cite{CC}.
There is just one point where we need to be careful when
following the proof of Theorem 9.7 in \cite{CC}:
since $F$ depends also on $\xi$,
 we cannot exploit the a priori estimates used therein (which are those given in Theorem 9.5 in \cite{CC}).  In place, we can invoke the a priori estimates given in \cite[Theorem 17.26]{GT}. These estimates are stated actually for more regular solutions, but this is not restrictive thanks to classical Schauder estimates, which hold in particular by  the  $C ^ {2, \beta}$ regularity of $\partial \Omega$ (see \cite[Section 6.4]{GT}). The relevant point is that the estimates in \cite[Theorem 17.26]{GT}  continue to hold for $F = F (\xi, X)$,  and enable us to conclude along the proof line of \cite[Theorem 9.7]{CC}.
As a drawback, we have to ask the convexity condition in the reinforced form (H3)', which is needed precisely to ensure the validity of condition (17.85) in \cite{GT}. 
\qed

\subsection{Proof of Theorem~\ref{t:log}(ii)}
Let $\Omega\in\clset$,
and let $\{u_k\}$ be the
approximating sequence given by Theorem~\ref{t:approx}.
We remark that the approximating smooth sets $\{\Omega_k\}$
can be chosen to be strongly convex.
Since, by Theorem~\ref{t:log}(i), every function $u_k$ is log-concave, then also their uniform limit $u$ is a log-concave positive eigenfunction.
\qed

\bigskip{\bf Acknowledgments.} We wish to thank Bernd Kawohl for some useful comments about the validity of claim (b) in
the proof of  Theorem \ref{t:log}, and Isabeau Birindelli for several interesting discussions.

\def\cprime{$'$}

\end{document}